 \documentclass[11pt]{amsart}

\usepackage{amssymb}
\usepackage{bbm}
\usepackage{anysize}
\marginsize{1in}{1in}{1in}{1in}

\usepackage{pdfsync}
\usepackage{color}
\newtheorem{tm}{Theorem}[section]

\newtheorem{defin}[tm]{Definition} 
\newtheorem{rk}[tm]{Remark}

\numberwithin{equation}{section}

\begin{document}


\title{Stochastic Solution of Fractional Fokker-Planck Equations with Space-Time-Dependent Coefficients}


\author{ERKAN NANE}
\address{Department of Mathematics and Statistics,
Auburn University,
Auburn, AL 36849 USA}
\email{ezn0001@auburn.edu}

\author{YINAN NI}
\address{Department of Mathematics and Statistics,
Auburn University,
Auburn, AL 36849 USA}
\email{yzn0005@auburn.edu}



\begin{abstract}
This paper develops solutions of fractional Fokker-Planck equations describing subdiffusion of probability densities of stochastic dynamical systems driven by non-Gaussian L\'evy processes, with space-time-dependent drift, diffusion and jump coefficients, thus significantly extends Magdziarz and Zorawik's result in \cite{mazo}. Fractional Fokker-Planck equation describing subdiffusion is solved by our result in full generality from perspective of stochastic representation.
\end{abstract}

\maketitle


\section{Introduction}

Fractional Fokker-Planck equation has shown its application in diverse scientific areas, including biology \cite{bioapp}, physics \cite{phyapp}, \cite{sokl}, finance \cite{finapp}. For example, in physics, it has been broadly used to describe phenomena related to anomalous diffusion \cite{phyappss}, \cite{phyapps}.    Many different types of fractional Fokker-Planck equation have been solved in terms of probability density function (PDF) of corresponding process \cite{physic}, \cite{mekl}, related researches are also growing rapidly in different branches, e.g., \cite{mesi}, \cite{mena}, \cite{menave}, \cite{menav}, \cite{mina}. Recently, Magdziarz and Zorawik \cite{mazo} provide solution to an extended type of Factional Fokker-Planck equation.

To state the result, let's introduce subordinator $T_\Psi(t)$ with Laplace transform
\begin{equation}
E(e^{-uT_\Psi(t)})=e^{-t\Psi(u)},
\end{equation}
where Laplace exponent
\begin{equation}\label{laplacetransform}
\Psi(u)=\int_0^{\infty}(1-e^{-ux})\nu(dx),
\end{equation}
$\nu$ is L\'evy measure satisfying
\begin{equation}
\int_{\mathbf{R}-\{0\}}\min\{|x|^2,1\}\nu(dx)<\infty
\end{equation} and $\nu((0,\infty))=\infty$.
 Then, the first passenge time process defined as
 \begin{equation}
 S_{\Psi}(t)=\inf\{\gamma>0: T_\Psi(\gamma)>t\}, t\geq 0,
 \end{equation}
 is called the inverse subordinator.

 Define the integro-differential operator $\Phi_t$ as
 \begin{equation}
 \Phi_tf(t)=\frac{\partial}{\partial t} \int_0^t M(t-y)f(y)dy,
 \end{equation}
 where the function $f$ is smooth enough and kernel $M(t)$ is defined by its Laplace transform as
\begin{equation}
\tilde{M}(u)=\int_0^{\infty}e^{-ut}M(t)dt=\frac{1}{\Psi(u)}.
\end{equation}

A stochastic process $X=\{X(t),t\geq 0\}$ is a L\'evy process if (a) $X(0)=0,\ a.s.$, (b) $X$ has independent and stationary increments, (c) $X$ is stochastically continuous in time. $X^-(t)$ is used to denote left limit, $X^-(t)=\lim\limits_{s\rightarrow t-}X(s)$.

By Theorem 1.2.14 and Proposition 1.3.1 of \cite{appl}, a L\'evy process $X$ has characteristics $(b,A,\nu)$, that's,
\begin{equation}
E(e^{iuX(t)})=\exp \left[t \left( ibu-\frac{1}{2}Au^2+\int_{\mathbf{R}-\{0\}}(e^{iuy}-1-iuy\mathbbm{1}_B(y) )\nu(dy) \right) \right],
\end{equation}
where $\nu$ is L\'evy measure. In the remaining part of this paper, for convenience,  we will decompose a L\'evy process $X(t)$ into 3 parts: drift, Brownian motion $B(t)$ and pure jump L\'evy process $L(t)$.

Improving on the methods in the papers \cite{henry} \cite{magd} \cite{maga},   Magdziarz and Zorawik \cite{mazo} proved that the PDF of process $X(t)=Y^-(S_{\Psi}(t))$, where
\begin{equation}
\begin{aligned}
&dY(t)=F(Y^-(t),T^-_\Psi(t))dt+\sigma(Y^-(t),T^-_\Psi(t))dB(t)+E(T^-_\Psi(t))dL(t), t\geq 0,\\
&Y(0)=0, T_\Psi(t)=0,
\end{aligned}
\end{equation}
with $F(x,t),\sigma(x,t),E\in C^2(\mathbf{R}^2)$ satisfying Lipschitz condition, solves fractional Fokker-Planck equation
\begin{equation}\label{magequ}
\begin{aligned}
\frac{\partial q(x,t)}{\partial t}&=\left[-\frac{\partial}{\partial x} F(x,t)+\frac{1}{2}\frac{\partial^2}{\partial x^2}\sigma^2(x,t)\right]\Phi_t q(x,t)\\
&+\int_{\mathbf{R}-\{0\}}\left[ \Phi_t q(r,t)|_{r=x-E(t)y}-\Phi_t q(x,t)+E(t)y\frac{\partial}{\partial x}\Phi_t q(x,t)\mathbbm{1}_B(y)\right] \nu(dy),
\end{aligned}
\end{equation}
with $q(x,0)=\delta(x)$, where $\mathbbm{1}$ is an indicator function, $B=\{y,|y|<1\}$ .

This result extends the following celebrated fractional Fokker-Planck equation introduced by Metzler and Klafter \cite{mekl} in 2000,
\begin{equation}\label{meklequ}
\frac{\partial q(x,t)}{\partial t} = ~_0D_t^{1-\alpha}\left[-\frac{\partial}{\partial x}F(x)+\frac{\sigma^2}{2}\frac{\partial^2}{\partial x^2}\right]q(x,t),
\end{equation}
with $\sigma>0$ and $q(x,0)=\delta(x)$, which describes anomalous diffusion in the presence of an space-dependent force $F(x)$. Note that the operator $~_0D_t^{1-\alpha},\ \alpha\in(0,1)$ is fractional derivative of Riemann-Liouville type \cite{saki},
\begin{equation}
~_0D_t^{1-\alpha}f(t)=\frac{1}{\Gamma(\alpha)}\frac{d}{dt}\int_0^t(t-s)^{\alpha-1}f(s)ds,
\end{equation}
for $f\in C^1([0,\infty))$.
Magdziarz et al. \cite{mawe} showed that the PDF of the subordinated process $X(t)=Y(S_{\alpha}(t))$
solves equation \eqref{meklequ}, where
\begin{equation}
dY(t)=F(Y(t))dt+\sigma dB(t), Y(0)=0.
\end{equation}

Here, let $D_\alpha$ be an $\alpha$-stable subordinator with  Laplace transform $E[e^{-uD_{\alpha}(\gamma)}]=e^{-tu^{\alpha}}$, its inverse $S_{\alpha}(t)$ is defined as
\begin{equation}
S_{\alpha}(t)=\inf\{ \gamma>0: D_{\alpha}(\gamma)>t\}.
\end{equation}

Magdziarz and Zorawik's result \cite{mazo} also extends the following fractional Fokker-Planck equation introduced by Sokolov and Klafter \cite{sokl} in 2009,
\begin{equation}\label{soklequ}
\frac{\partial q(x,t)}{\partial t} = \left[-F(t)\frac{\partial}{\partial x}+\frac{\sigma^2}{2}\frac{\partial^2}{\partial x^2}\right]~_0D_t^{1-\alpha}q(x,t),
\end{equation}
with $\sigma>0$ and $q(x,0)=\delta(x)$, where external force $F(t)$ is time-dependent. Its solution is obtained by PDF of subordinated process $X(t)=Y(S_\alpha(t))$, where
\begin{equation}
dY(t)=F(T_\alpha(t))dt+\sigma dB(t), Y(0)=0.
\end{equation}

One of the main results of this paper is that the PDF of $X(t)=Y^-(S_{\Psi}(t))$, where
\begin{equation}
\begin{aligned}
&dY(t)=F(Y^-(t),T^-_\Psi(t))dt+\sigma(Y^-(t),T^-_\Psi(t))dB(t)+h(Y^-(t),T^-_\Psi(t))dL(t), t\geq 0,\\
&Y(0)=0, T_\Psi(t)=0,
\end{aligned}
\end{equation}
solves the following fractional Fokker-Planck equation, which involves fractional Laplacian operator,
\begin{equation}\label{stable}
\frac{\partial q(x,t)}{\partial t}=\left[-\frac{\partial}{\partial x} F(x,t)+\frac{1}{2}\frac{\partial^2}{\partial x^2}\sigma^2(x,t)-(-\Delta)^{\frac{\alpha}{2}}(\text{sgn}(h(x,t))|h(x,t)|^{\alpha})\right]\Phi_t q(x,t),
\end{equation}
with $q(x,0)=\delta(x)$,
note that $\alpha \in (0,2)$ and $L(t)$ is the $\alpha$-stable  L\'evy process.

Furthermore, in section 3, we extend Magdziarz and Zorawik's result \cite{mazo} by solving
\begin{equation}\label{FFPEI}
\frac{\partial}{\partial t}w(x,t)=\left[-\frac{\partial}{\partial x}(F(x,t)+\frac{\partial^2}{\partial x^2}\frac{1}{2}\sigma^2(x,t)+T^{l+}\right]\Phi_t w(x,t),\\
\end{equation}
with $q(x,0)=\delta(x)$,
where
\begin{equation}
T^{l+} f(x,t)=\int_{\mathbf{R}-\{0\}}\left[ \sum_{k=1}^{\infty}\frac{(-r)^k}{k!}\frac{\partial^k}{\partial y^k}(h(x,t)^k f(x,t) )+\mathbbm{1}_B(r,t)r\frac{\partial}{\partial y}(h(x,t)f(x,t))\right]\nu(dr),
\end{equation}
for any $f(x,t)\in C_0^\infty(R^2)$.

Note that coefficient $E(t)$ of pure jump L\'evy process in \eqref{magequ} is time-dependent, while coefficient $h(x,t)$ of pure jump L\'evy process in \eqref{FFPEI} is space-time-dependent. Thus, \eqref{magequ} is a special case of \eqref{FFPEI} when $h(x,t)$ only depends on time $t$. For more details on this, see remark \ref{reduced}.

In the remaining of this paper, necessary concepts will be given in the Preliminaries section; in the Main Results section, we will solve 3 different fractional Fokker-Planck equations involving operators of $\alpha$-stable, symmetric and general L\'evy processes, respectively, one by one.

\section{Preliminaries}
Let $X=\{X(t),t\geq 0\}$ be a L\'evy process with characteristics $(b,A,\nu)$,  by Theorem 6.7.4 of \cite{appl}, it has infinitesimal generator
\begin{equation}
Af(x)=b\frac{\partial}{\partial x} f(x)+\frac{1}{2}\frac{\partial^2}{\partial x^2} f(x)+ \int_{\mathbf{R}-\{0\}}\left[f(x+y)-f(x)-y\frac{\partial}{\partial x}f(x)\mathbbm{1}_B(y)\right]\nu(dy),
\end{equation}
for each $f\in C_0^2(\mathbf{R}),\ x\in \mathbf{R}$.

The following L\'evy processes and their generators will be used in this paper later.

An $\alpha$-stable L\'evy process $X(t)$ has characteristics $(0,0,\nu)$ and L\'evy symbol $\eta(u)=-|u|^\alpha$, $\alpha\in (0,2)$, see example 3.3.8 of \cite{appl}, and infinitesimal generator
\begin{equation}
Af(x)=\int_{\mathbf{R}-\{0\}}[f(x+y)-f(x)]\nu(dy)=-(-\Delta)^{\alpha/2}f(x),
\end{equation}
where $\nu(dy)= \frac{C_\alpha dy}{|y|^{1+\alpha}},\ C_\alpha =\frac{\alpha}{\Gamma(1-\alpha)}$.

A symmetric L\'evy process $X(t)$ has characteristics $(0,0,\nu)$ where $\nu$ is symmetric L\'evy measure, that's, $\nu(B)=\nu(-B)$, for each $B \subset \mathbf{R}$, and infinitesimal generator
\begin{equation}
Af(x)=\int_{\mathbf{R}-\{0\}}[f(x+y)-f(x)]\nu(dy).
\end{equation}

A general L\'evy process $X(t)$ with characteristics $(0,0,\nu)$ has infinitesimal generator
\begin{equation}
Af(x)=\int_{\mathbf{R}-\{0\}}\left[f(x+y)-f(x)-y\frac{\partial}{\partial x}f(x)\mathbbm{1}_B(y)\right]\nu(dy),
\end{equation}
in fact, such $X(t)$ is a pure jump L\'evy process since drift and Brownian motion parts are gone as $b=A=0$.

Let $G(w)=\nu((w,\infty))$, define the following operator \cite{mazo}
\begin{equation}
\Theta_wg(w)=\int_0^wG(w-z)g(z)dz,
\end{equation}
with Laplace transform of its kernel
\begin{equation}
\tilde{G}(u)=\frac{\Psi(u)}{u}.
\end{equation}

\section{Main Results}

In this section, we will solve fractional Fokker-Planck equations with infinitesimal generator of L\'evy processes with space-time-dependent coefficients for drift, diffusion and jump parts.
We analyze the cases when the jump process is  $\alpha$-stable and symmetric L\'evy process, respectively, before analyzing general L\'evy process, since the results for the cases of $\alpha$-stable and symmetric L\'evy process are more explicit.

\subsection{Fractional Fokker-Planck equation with $\alpha$-stable L\'evy generator}

\begin{defin}\label{stable-sde}

Let $L(t)$ be  $\alpha$-stable  L\'evy process satisfying the following SDE
\begin{equation}\label{stablecon}
dL(t)=\int_B x\tilde{N} (dt,dx)+\int_{B^c} xN(dt,dx),
\end{equation}
where L\'evy symbol $ \eta(u)=\int_{\mathbf{R}-\{0\}} [e^{iux}-1-iux \mathbbm{1}_B(x)]\nu(dx),$ with $$ \nu(dx)=C_{\alpha}\frac{dx}{|x|^{1+\alpha}},\ B=\{x,|x|<1\}.$$\\
\end{defin}

\begin{tm}\label{thmstable}
Suppose that the standard Brownian motion $B(t)$, the $\alpha$-stable  L\'evy process $L(t)$ as defined in definition \ref{stable-sde} and subordinator $T_\Psi(t)$ are independent, where $T_\Psi(t)$ has Laplace exponent $\Psi(u)$ and its inverse $S_\Psi(t)$. Let $Y(t)$ be the solution of the stochastic equation

\begin{equation}\label{solofstable}
\begin{aligned}
dY(t)&=F(Y^-(t),T^-_\Psi(t))dt+\sigma(Y^-(t),T^-_\Psi(t))dB(t)+h(Y^-(t),T^-_\Psi(t))dL(t), t\geq 0,\\
Y(0)&=0,\ T_\Psi(0)=0
\end{aligned}
\end{equation}
where the function $F(x,t),\ \sigma(x,t),\ h(x,t)\in C^2(\mathbf{R}^2)$ satisfy the Lipschitz condition.
Assume that the PDF of the process $(Y(t),T_\Psi(t)),\ p_t(y,z) $ exists. Furthermore, we assume that $\frac{\partial}{\partial t}p_t(y,z),\ \frac{\partial}{\partial y}p_t(y,z)$, $\frac{\partial^2}{\partial y^2}p_t(y,z)$ exist,
\begin{equation}\label{condition11}
\int_0^{t_0}\int_0^{\infty}\left|\frac{\partial}{\partial t}p_s(x,t)\right|dsdt<\infty
\end{equation}
for each $t_0>0$,
\begin{equation}
\int_{x_1}^{x_2}\int_0^{\infty}\left|\frac{\partial}{\partial x}p_s(x,t)\right|dsdx<\infty,
\int_{x_1}^{x_2}\int_0^{\infty}\left|\frac{\partial^2}{\partial x^2}p_s(x,t)\right|dsdx<\infty,
\end{equation}
for each $x_1, x_2\in \mathbf{R}$ and
\begin{equation}
\int_0^{\infty}\int_{\mathbf{R}-\{0\}}\left|p_s(x+y,t)\frac{\text{sgn}(h(x+y,t))|h(x+y,t)|^{\alpha}}{|y|^{1+\alpha}}-p_s(x,t)\frac{\text{sgn}(h(x,t))|h(x,t)|^{\alpha}}{|y|^{1+\alpha}}\right|C_{\alpha}dyds<\infty.
\end{equation}
for each $x\in \mathbf{R}$.Then, the PDF of the process $X(t)=Y^-(S_\Psi(t))$ is the weak solution of fractional Fokker-Planck equation \eqref{stable}.
\end{tm}

\begin{proof}  This proof uses methods in \cite{henry} and \cite{mazo} with crucial changes.

Equation \eqref{solofstable} can be represented as the following stochastic equations
\begin{equation}
\begin{aligned}
dY(t)&=F(Y^-(t),Z^-(t))dt+\sigma(Y^-(t),Z^-(t))dB(t)\\ &+h(Y^-(t),Z^-(t)) \int_B x\tilde{N} (dt,dx)+ h(Y^-(t),Z^-(t)) \int_{B^c}x  N(dt,dx)\\
dZ(t)&=dT_\Psi(t).
\end{aligned}
\end{equation}

By Theorem 6.7.4 of \cite{appl}, the infinitesimal generator of the process $(Y(t),Z(t))$ that operates on functions $f\in C^2_0(\mathbf{R}^2)$ is
\begin{equation}
\begin{aligned}
\Gamma f(y,z)=&F(y,z)\frac{\partial}{\partial y}f(y,z)+\frac{1}{2}\sigma^2(y,z)\frac{\partial^2}{\partial y^2}f(y,z)\\
+&\int_{\mathbf{R}-\{0\}}\left[f(y+xh(y,z),z)-f(y,z)-\mathbbm{1}_B(x)xh(y,z)\frac{\partial}{\partial y}f(y,z)\right]\nu(dx)\\
+&\int_0^{\infty}[f(y,z+u)-f(y,z)]\mu(du),
\end{aligned}
\end{equation}
where $\mu$ is L\'evy measure of $T_\Psi(t)$.

Decompose $\Gamma=A+T_\alpha$, where
\begin{equation}
\begin{aligned}
Af(y,z)=&F(y,z)\frac{\partial}{\partial y}f(y,z)+\frac{1}{2}\sigma^2(y,z)\frac{\partial^2}{\partial y^2}f(y,z),\\
+&\int_0^{\infty}[f(y,z+u)-f(y,z)]\mu(du)\\
T_\alpha f(y,z)=&\int_{\mathbf{R}-\{0\}}\left[f(y+xh(y,z),z)-f(y,z)-\mathbbm{1}_B(x)xh(y,z)\frac{\partial}{\partial y}f(y,z)\right]\nu(dx).
\end{aligned}
\end{equation}

By setting $E(t)=0$ in (2.17) of \cite{mazo},
 \begin{equation}
 A^+p_t(y,z)=-\frac{\partial}{\partial y}(F(y,z)p_t(y,z))+\frac{\partial^2}{\partial y^2}(\frac{1}{2}\sigma^2(y,z)p_t(y,z))-\frac{\partial}{\partial z}\Theta_z p_t(y,z),
 \end{equation}
 where $A^+$ is the Hermitian adjoint of $A$ .

Next, $T_\alpha^+$ is derived below
\begin{equation}\label{alpha}
\begin{aligned}
T_\alpha f(y,z)=&\int_{\mathbf{R}-\{0\}}[f(y+xh(y,z),z)-f(y,z)]C_{\alpha}\frac{dx}{|x|^{1+\alpha}}\\
=&\text{sgn}(h(y,z))|h(y,z)|^{\alpha}\int_{\mathbf{R}-\{0\}}[f(y+x,z)-f(y,z)]C_{\alpha}\frac{dx}{|x|^{1+\alpha}}\\
=&\text{sgn}(h(y,z))|h(y,z)|^{\alpha}[-(-\Delta)^{\alpha/2}f(y,z)].\\
\end{aligned}
\end{equation}

Note that the second equation is the result of change of variable.

 Since the infinitesimal generator of $\alpha$-stable L\'evy process is self-adjoint,
 \begin{equation}
 \begin{aligned}
 &\int_{\mathbf{R}^2}f(y,z)T_\alpha^+p_t(y,z)dydz\\
 =&\int_{\mathbf{R}^2} p_t(y,z)T_\alpha f(y,z)dydz\\
 =&\int_{\mathbf{R}^2}p_t(y,z)\text{sgn}(h(y,z))|h(y,z)|^{\alpha}[-(-\Delta)^{\alpha/2}f(y,z)]dydz\\
 =&\int_{\mathbf{R}^2}f(y,z) [-(-\Delta)^{\alpha/2}(p_t(y,z)\text{sgn}(h(y,z))|h(y,z)|^{\alpha})]dydz.
 \end{aligned}
 \end{equation}

 That's,
 \begin{equation}
 T_\alpha^+p_t(y,z)=-(-\Delta)^{\alpha/2}(p_t(y,z)\text{sgn}(h(y,z))|h(y,z)|^{\alpha})
 \end{equation}

Since $\Gamma^+=A^++T_\alpha^+$ and $\frac{\partial}{\partial t}p_t(y,z)=\Gamma^+p_t(y,z)$, we have
\begin{equation}\label{alphaopcon}
\begin{aligned}
\frac{\partial}{\partial t}p_t(y,z)=&\Gamma^+p_t(y,z)=A^+p_t(y,z)+T_\alpha^+p_t(y,z)\\
=&-\frac{\partial}{\partial y}(F(y,z)p_t(y,z))+\frac{\partial^2}{\partial y^2}(\frac{1}{2}\sigma^2(y,z)p_t(y,z))-\frac{\partial}{\partial z}\Theta_z p_t(y,z)\\
&-(-\Delta)^{\alpha/2}(p_t(y,z)\text{sgn}(h(y,z))|h(y,z)|^{\alpha}).
\end{aligned}
\end{equation}

Next, we establish the relationship between $q(x,t)$ and $p_t(y,z)$, the probability density functions of $X(t)$ and $(Y(t), Z(t))$, respectively. Let $w$ denote a random path of stochastic process, for each fixed interval $I$, define indicator function
\begin{equation}
    \mathbbm{1}_I(x)=
   \begin{cases}
   1 &\mbox{if $x\in I$,}\\
   0 &\mbox{otherwise.}
   \end{cases}
  \end{equation}
and the auxiliary function \cite{henry}
\begin{equation}
    H_t(s,w,u)=
   \begin{cases}
   \mathbbm{1}_I(Y^-(s,w)) &\mbox{if $Z^-(s,w)\leq t \leq Z^-(s,w)+u$,}\\
   0 &\mbox{otherwise.}
   \end{cases}
  \end{equation}
Then, we get
\begin{equation}\label{second1}
\int_I q(x,t)dx=E[\mathbbm{1}_I(X(t,w))].
\end{equation}

Let $\Delta Z(t,w)=Z(t,w)-Z^-(t,w)$ and $s=S_\Psi(t,w)=\inf\{\tau>0:T_\Psi(\tau)>t\}$, then $H_t(S_\Psi(t),w,\Delta Z(t,w))=\mathbbm{1}_I(X(t,w))$. We can prove this as follows, by definition of $H_t(s,w,u)$,
\begin{equation}
H_t(S_\Psi(t),w,\Delta Z(t,w))=
   \begin{cases}
   \mathbbm{1}_I(Y^-(S_\Psi(t,w),w)) &\mbox{if $Z^-(S_\Psi(t,w),w)\leq t \leq Z(S_\Psi(t,w),w)$,}\\
   0 &\mbox{otherwise.}
   \end{cases}
\end{equation}
Since $T_\Psi(t,w)$ is a subordinator, $T^-_\Psi(S_\Psi(t,w),w)\leq t \leq T_\Psi(S_\Psi(t,w),w)$ is always true, thus
\begin{equation}
H_t(S_\Psi(t,w),w,\Delta Z(t,w))=\mathbbm{1}_I(Y^-(S_\Psi(t,w),w)).
\end{equation}

To avoid $Z(t,w)$ being a compound Poisson process, we set $\nu([0,\infty)=\infty$, see Remark 27.3 and 27.4 of \cite{sato}, thus, jumping times of $Z(t,w)$ are dense in $[0,\infty]$ almost surely, see Theorem 21.3 of \cite{sato}. Then, we can derive that $$H_t(s,w,\Delta Z(t,w))=0,\ if\ s\neq S_\Psi(t,w)$$
Since if $s<S_\Psi(t,w)$, as $Z(t)=T_\Psi(t)$ is a subordinator, $Z(s)<Z^-(S_\Psi(t,w))\leq t.$
Similarly, if $s>S_\Psi(t,w)$, then $Z(s)>Z(S_\Psi(t,w))\geq t$. In both cases, $H_t(s,w,\Delta Z(t,w))=0$.

Hence,
\begin{equation}\label{second2}
\mathbbm{1}_I(X(t,w))=\sum\limits_{s>0}H_t(s,w,\Delta Z(t,w)).
\end{equation}
By Compensation Formula in Ch. XII, Proposition (1.10) of \cite{yor},
\begin{equation}\label{second3}
\begin{aligned}
 &E\left[\sum\limits_{s>0}H_t(s,w,\Delta Z(s,w))\right]= E\left[\int_0^{\infty} \int_0^{\infty} H_t(s,w,u)\nu(du)ds\right]\\
=&E\left[\int_0^{\infty}\int_0^{\infty}\mathbbm{1}_I(Y(s,w))\mathbbm{1} _{[Z(s-,w),Z(s-w)+u]}(t)\nu(du) ds\right]\\
=&E\left[\int_0^{\infty} \mathbbm{1}_I(Y(s,w))\int_0^{\infty}\mathbbm{1} _{[t-z,\infty]}(u)\mathbbm{1} _{[0,t]}(z)\nu(du) ds\right]\\
=&E\left[\int_0^{\infty} \mathbbm{1}_I(Y(s,w))\mathbbm{1} _{[0,t]}(Z(s,w))\nu(t-Z(s,w),\infty) ds\right]\\
=&E\left[\int_0^{\infty} \mathbbm{1}_I (Y(s,w))\mathbbm{1} _{[0,t]}(Z(s,w))G(t-Z(s,w)ds\right]\\
=&\int_I \int_0^{\infty} \int_0^t G(t-z)p_s(y,z)dzdsdy\\
=&\int_I \int_0^{\infty} \Theta_t p_s(y,z)dsdy.
\end{aligned}
\end{equation}
By \eqref{second1}, \eqref{second2} and \eqref{second3}, we have
\begin{equation}
\int_I q(x,t)dx=\int _I\int_0^{\infty} \Theta_t p_s(y,t)dsdy.
\end{equation}
By the arbitrariness of interval $I$,
\begin{equation}\label{qintp}
q(x,t)=\int_0^{\infty} \Theta_t p_s(y,t)ds.
\end{equation}

Next we claim that
\begin{equation}\label{exchange}
\frac{\partial}{\partial t}\Theta_t p_s(x,t)=\Theta_t \frac{\partial}{\partial t} p_s(x,t).
\end{equation}
To see this, let $P_s(x,u)$ and $G(u)$ be the Laplace transform ($t \to u$) of $p_s(x,t)$ and $g(t)$, respectively.
Then Laplace transform of $\frac{\partial}{\partial t}\Theta_t p_s(x,t)$ is given as
\begin{equation}
\begin{aligned}
\mathcal{L}\left[\frac{\partial}{\partial t}\Theta_t p_s(x,t)\right]=&u\mathcal{L}[\Theta_t p_s(x,t)]-\Theta_0 p_s(x,0)\\
=&u\mathcal{L}\left[\int_0^t G(t-z)p_s(x,z)dz\right]\\
=&uG(u)P_s(x,u).
\end{aligned}
\end{equation}
On the other hand, since $T_\Psi(0)=0$ a.s. , Laplace transform of $\Theta_t \frac{\partial}{\partial t} p_s(x,t)$ is as below,
\begin{equation}
\begin{aligned}
\mathcal{L}\left[\Theta_t \frac{\partial}{\partial t} p_s(x,t)\right]=&\mathcal{L}\left[\int_0^t G(t-z) \frac{\partial}{\partial z}p_s(x,z)dz\right]\\
=&G(u)\mathcal{L}\left[\frac{\partial}{\partial t}p_s(x,t)\right]\\
=&G(u)\{ uP_s(x,u)-p_s(x,0) \}\\
=&uG(u)P_s(x,u).
\end{aligned}
\end{equation}

Notice that
\begin{equation}
\int_0^{t_0}G(u)du=\int_0^{t_0}\int_{(u,\infty)}\nu(dw)du=\int_{(0,\infty)}\min(w,t_0)\nu(dw)=K<\infty,
\end{equation}
by assumption \eqref{condition11} and \eqref{exchange}, we derive that
\begin{equation}
\begin{aligned}
\int_0^{t_0}\int_0^{\infty}\left|\frac{\partial}{\partial t}\Theta_t p_s(x,t)\right|dsdt=&\int_0^{t_0}\int_0^{\infty}\left|\Theta_t \frac{\partial}{\partial t}p_s(x,t)\right|dsdt\\
=&\int_0^{t_0}\int_0^{\infty}\left|\int_0^t G(t-u) \frac{\partial}{\partial z}p_s(x,z)\right|dsdt\\
\leq & \int_0^{t_0}\int_0^{\infty}\int_0^t G(t-u) \left|\frac{\partial}{\partial z}p_s(x,z)\right|dudsdt\\
=&\int_0^{t_0}\int_0^{\infty}\int_u^t G(t-u) \left|\frac{\partial}{\partial z}p_s(x,z)\right|dtdsdu\\
=&\int_0^{t_0}\int_0^{\infty}\left|\frac{\partial}{\partial z}p_s(x,z)\right| \int_u^t G(t-u) dtdsdu\\
\leq& K \int_0^{t_0}\int_0^{\infty}\left|\frac{\partial}{\partial z}p_s(x,z)\right|dsdu\\
<&\infty
\end{aligned}
\end{equation}

Thus, we can put differentiation on both side of \eqref{qintp} and move it inside the integral on the righthand side as below,
\begin{equation}\label{equ1}
\frac{\partial}{\partial t} q(x,t)=\int_0^\infty \frac{\partial}{\partial t}\Theta_t p_s(x,t)ds.
\end{equation}

Next, we claim that
\begin{equation}\label{equ2}
\int_0^\infty p_s(x,t)ds=\Phi_tq(x,t).
\end{equation}

By Fubini theorem,
\begin{equation}
\begin{aligned}
q(x,t)=&\int_0^\infty \Theta_t p_s(x,t)ds=\int_0^\infty \int_0^t G(t-z)p_s(x,z)dzds\\
=&\int_0^t  G(t-z)\int_0^\infty p_s(x,z)dsdz=\Theta_t\int_0^\infty p_s(x,t)ds,
\end{aligned}
\end{equation}
thus,
\begin{equation}
\int_0^\infty p_s(x,t)ds=\Theta_t^{-1}q(x,t).
\end{equation}
To prove $\Theta _t^{-1}=\Phi_t$, let $\tilde{ M}(u)$, $\tilde{G}(u)$ and $\tilde{Q}(x,u)$ be the Laplace transform of $M(t)$, $G(t)$ and $q(x,t)$, respectively.
Since, by \eqref{laplacetransform}
\begin{equation}
\int_0^\infty e^{-ut}G(t)dt=\int_0^\infty \int_{(t,\infty)}e^{-ut}\nu(ds)dt=\int_0^\infty \frac{1-e^{us}}{u}\nu(ds)=\frac{\Psi(u)}{u}, \\
\end{equation}
and
\begin{equation}
\mathcal{L}[q(x,z)]=\mathcal{L}[G(t)]\mathcal{L}[\Theta _t^{-1}q(x,t)],
\end{equation}
we have
\begin{equation}
\begin{aligned}
\mathcal{L}[\Theta _t^{-1}q(x,t)]=&\frac{\mathcal{L}[q(x,z)]}{\mathcal{L}{[G(t)]}}=\frac{\tilde{Q}(x,u)}{\tilde{G}(u)}=u\frac{\tilde{Q}(x,u)}{\Psi(u)}.
\end{aligned}
\end{equation}
Also
\begin{equation}
\begin{aligned}
\mathcal{L}[\Phi _t q(x,t)]=&\mathcal{L}\left[\frac{d}{dt}\int_0^t M(t-y)q(x,y)dy\right]\\
=&u\mathcal{L}\left[\int_0^t M(t-y)q(x,y)dy \right]-0\\
=&u\tilde{M}(u)\tilde{Q}(x,u)=u\frac{\tilde{Q}(x,u)}{\Psi(u)},
\end{aligned}
\end{equation}
this shows $\Phi_t q(x,t)=\Theta_t^{-1}q(x,t)$, hence $\int_0^\infty p_s(x,t)ds=\Phi_tq(x,t)$.

Since $\lim\limits_{s\rightarrow \infty} p_s(x,t)=0$ and $p_0(x,t)=\mathbbm{1}_{(0,0)}(x,t)$, by \eqref{alphaopcon}, \eqref{equ1} and \eqref{equ2},
\begin{equation}
\begin{aligned}
\frac{\partial}{\partial t} q(x,t)=&\int_0^{\infty} \left[ \frac{\partial^2}{\partial x^2}(\frac{1}{2}\sigma^2(x,t)p_s(x,t))-\frac{\partial}{\partial x}(F(x,t)p_s(x,t))\right.\\
&\left. -(-\Delta)^{\alpha/2}(p_s(x,t)(\text{sgn}(h(x,t))|h(x,t)|^{\alpha}))-\frac{\partial}{\partial s}p_s(x,t)\right]ds\\
=&\int_0^{\infty}\left[ \frac{\partial^2}{\partial x^2}(\frac{1}{2}\sigma^2(x,t)p_s(x,t))-\frac{\partial}{\partial x}(F(x,t)p_s(x,t))\right]ds\\
&+\int_0^{\infty}\int_{\mathbf{R}-\{0\}}\left(p_s(x+y,t)\frac{\text{sgn}(h(x+y,t))|h(x+y,t)|^{\alpha}}{|y|^{1+\alpha}}\right.\\
&\left.-p_s(x,t)\frac{\text{sgn}(h(x,t))|h(x,t)|^{\alpha}}{|y|^{1+\alpha}}\right)C_{\alpha}dyds\\
=&\int_0^{\infty}\left[\frac{\partial^2}{\partial x^2}(\frac{1}{2}\sigma^2(x,t)p_s(x,t))-\frac{\partial}{\partial x}(F(x,t)p_s(x,t))\right]ds\\
&+\int_{\mathbf{R}-\{0\}}\left[ \int_0^{\infty}p_s(x+y,t)ds\frac{\text{sgn}(h(x+y,t))|h(x+y,t)|^{\alpha}}{|y|^{1+\alpha}}\right.\\
&\left.-\int_0^{\infty}p_s(x,t)ds\frac{\text{sgn}(h(x,t))|h(x,t)|^{\alpha}}{|y|^{1+\alpha}}\right]C_{\alpha}dy\\
=&\frac{\partial^2}{\partial x^2}\left(\frac{1}{2}\sigma^2(x,t)\int_0^{\infty}p_s(x,t)ds\right)-\frac{\partial}{\partial x}\left(F(x,t)\int_0^{\infty}p_s(x,t)ds\right)\\
&-(-\Delta)^{\alpha/2}\left((\text{sgn}(h(x,t))|h(x,t)|^{\alpha})\int_0^{\infty}p_s(x,t)ds\right)\\
=&\frac{\partial^2}{\partial x^2}(\frac{1}{2}\sigma^2(x,t)\Phi_t q(x,t))-\frac{\partial}{\partial x}(F(x,t)\Phi_t q(x,t))\\
&-(-\Delta)^{\alpha/2}((\text{sgn}(h(x,t))|h(x,t)|^{\alpha})\Phi_t q(x,t)).
\end{aligned}
\end{equation}

In summary,
\begin{equation}
\frac{\partial}{\partial t} q(x,t)=\left[\frac{1}{2}\frac{\partial}{\partial x^2}\sigma^2(x,t)-\frac{\partial}{\partial x}F(x,t)-(-\Delta)^{\alpha/2}(\text{sgn}(h(x,t))|h(x,t)|^{\alpha}) \right]\Phi_t q(x,t)
\end{equation}
\end{proof}

\subsection{Fractional Fokker-Planck equation with symmetric L\'evy generator}

As is known, $\alpha$-stable L\'evy process is a special case of symmetric L\'evy process, this subsection solves fractional Fokker-Planck equation associated with symmetric L\'evy process with space-time-dependent coefficient, which extends result of previous subsection.
\begin{defin}\label{definitionofsymmetric}
Let $L(t)$ be a $symmetric \ L\acute{e}vy \  process$ satisfying the following SDE
\begin{equation}
dL(t)=\int_B x\tilde{N} (dt,dx)+\int_{B^c} xN(dt,dx)
\end{equation} where
$L\acute{e}vy\ symbol\ \eta(u)=\int_{\mathbf{R}} [cos(u,x)-1]\nu(dx)$ and $\nu$ is a symmetric $L\acute{e}vy $ measure.\\
\end{defin}

\begin{tm}\label{thmsymmetric}
Suppose that the standard Brownian motion $B(t)$, the symmetric  L\'evy process $L(t)$ as defined in definition  \ref{definitionofsymmetric} and subordinator $T_\Psi(t)$ are independent, where $T_\Psi(t)$ has Laplace exponent $\Psi(u)$ and its inverse $S_\Psi(t)$.
Let $Y(t)$ be the solution of the stochastic equation
\begin{equation}
\begin{aligned}
dY(t)&=F(Y^-(t),T^-_\Psi(t))dt+\sigma(Y^-(t),T^-_\Psi(t))dB(t)+h(Y^-(t),T^-_\Psi(t))dL(t), t\geq 0,\\
Y(0)&=0,\ T_\Psi(0)=0
\end{aligned}
\end{equation}
where the function $F(x,t),\ \sigma(x,t),\ h(x,t)\in C^2(\mathbf{R}^2)$ satisfy the Lipschitz condition.
Assume that the PDF of the process $(Y(t),T_\Psi(t)),\ p_t(y,z) $ exists. Furthermore, we assume that $\frac{\partial}{\partial t}p_t(y,z),\ \frac{\partial}{\partial y}p_t(y,z)$, $\frac{\partial^2}{\partial y^2}p_t(y,z)$ exist,
\begin{equation}
\int_0^{t_0}\int_0^{\infty}\left|\frac{\partial}{\partial t}p_s(x,t)\right|dsdt<\infty
\end{equation}
for each $t_0>0$,
\begin{equation}
\int_{x_1}^{x_2}\int_0^{\infty}\left|\frac{\partial}{\partial x}p_s(x,t)\right|dsdx<\infty,
\int_{x_1}^{x_2}\int_0^{\infty}\left|\frac{\partial^2}{\partial x^2}p_s(x,t)\right|dsdx<\infty,\\
\end{equation}
for each $x_1, x_2\in \mathbf{R}$ and
\begin{equation}
\int_0^{\infty}\int_{\mathbf{R}-\{0\}}|p_s(x+r,t)-p_s(x,t)|\nu'(dr)ds<\infty.
\end{equation}
for each $x\in \mathbf{R}$, where $\nu'(B)=\nu(\{x;xh(y,z)\in B\})$.

Then, the PDF of the process $X(T)=Y^-(S_\Psi(t))$ is the weak solution of fractional Fokker-Planck equation
\begin{equation}\label{symmetricFFPE}
\frac{\partial q(x,t)}{\partial t}=\left[-\frac{\partial}{\partial x} F(x,t)+\frac{1}{2}\frac{\partial^2}{\partial x^2}\sigma^2(x,t) + T^s\right]\Phi_t q(x,t),
\end{equation}
with $q(x,0)=\delta(x)$,
where
\begin{equation}
T^sf(x,t)=\int_{\mathbf{R}-\{0\}}[f(x+r,t)-f(x,t)]\nu'(dr),
\end{equation}
for any $f(x,t)\in C_0^2(\mathbf{R}^2)$.
\end{tm}

\begin{proof} We follow the same steps as in the proof of Theorem \ref{thmstable}, and modify the part related to $T_\alpha$.
Let  $L(t)$ be a symmetric L\'evy process as defined in definition \ref{definitionofsymmetric}, it has self-adjoint infinitesimal generator
\begin{equation}
(Tf)(y)=\int_{\mathbf{R}-\{0\}}[f(y+x)-f(y)]\nu(dx)
\end{equation}
for each $f\in C^2_0(R)$, and L\'evy symbol
\begin{equation}
\eta_l(u)=\int_{\mathbf{R}-\{0\}}[cos(u,x)-1]\nu(dx).
\end{equation}

Define $L_h(t)=L(t)h(Y(t),Z(t))$, since theorem \ref{thmsymmetric} involves $h(Y(t),T_\Psi(t))dL(t)$ instead of just $dL(t)$, by Proposition 11.10 of \cite{sato} and Corollary 3.4.11 of \cite{appl}, $L_h(t)$ has L\'evy symbol
\begin{equation}
\eta_{l_h}(u)=\int_{\mathbf{R}-\{0\}}[cos(u,x)-1]\nu'(dx),
\end{equation}
also $\nu'$ is symmetric, hence $L_h$ has self-adjoint generator
\begin{equation}
T^s f(y,z)=\int_{\mathbf{R}-\{0\}} [f(y+x,z)-f(y,z)]\nu'(dx).
\end{equation}

It follows that
\begin{equation}
\begin{aligned}
\int_{\mathbf{R}^2}p_t(y,z)T^sf(y,z)dydz&=\int_{\mathbf{R}^2}f(y,z)T^sp_t(y,z)dydz\\
&=\int_{\mathbf{R}^2}f(y,z)\int_{\mathbf{R}-\{0\}}[p_t(y+x,z)-p_t(y,z)]\nu'(dx)dydz,\\
\end{aligned}
\end{equation}
so,
\begin{equation}
T^{+}_s p_t(y,z)=\int_{\mathbf{R}-\{0\}}[p_t(y+x,z)-p_t(y,z)]\nu'(dx)
\end{equation}

Since $\Gamma^+=A^++T^{s+}$ and $\frac{\partial}{\partial t}p_t(y,z)=\Gamma^+p_t(y,z)$, we have
\begin{equation}
\begin{aligned}
\frac{\partial}{\partial t}p_t(y,z)&=\Gamma^+p_t(y,z)\\
=&A^+p_t(y,z)+T^{s+}p_t(y ,z)\\
=&-\frac{\partial}{\partial y}(F(y,z)p_t(y,z))+\frac{\partial^2}{\partial y^2}(\frac{1}{2}\sigma^2(y,z)p_t(y,z))-\frac{\partial}{\partial z}\Theta_z p_t(y,z)\\
&+\int_{\mathbf{R}-\{0\}}[p_t(y+x,z)-p_t(y,z)]\nu'(dx)
\end{aligned}
\end{equation}

Similar as proof of Theorem \ref{thmstable}, we have \eqref{equ1} and \eqref{equ2} , plus $\lim\limits_{s\rightarrow \infty} p_s(x,t)=0$ and $p_0(x,t)=\mathbbm{1}_{(0,0)}(x,t)$,

\begin{equation}
\begin{aligned}
\frac{\partial}{\partial t}q(x,t)=&\int_0^\infty\left[-\frac{\partial}{\partial x}(F(x,t)p_s(x,t))+\frac{\partial^2}{\partial x^2}(\frac{1}{2}\sigma^2(x,t)p_s(x,t))-\frac{\partial}{\partial s}p_s(x,t)\right]ds\\
&+\int_0^\infty \int_{\mathbf{R}-\{0\}}[p_s(x+r,t)-p_s(x,t)]\nu'(dr)ds\\
=&-\frac{\partial}{\partial x}\left[F(x,t)\int_0^\infty p_s(x,t)ds\right]+\frac{\partial^2}{\partial x^2}\left[\frac{1}{2}\sigma^2(x,t)\int_0^\infty p_s(x,t)ds\right]\\
&+\int_{\mathbf{R}-\{0\}}\left[\int_0^{\infty}p_s(x+r,t)ds-\int_0^{\infty}p_s(x,t)ds\right]\nu'(dr)\\
=&-\frac{\partial}{\partial x}\left[F(x,t)\int_0^\infty p_s(x,t))ds\right]+\frac{\partial^2}{\partial x^2}\left[\frac{1}{2}\sigma^2(x,t)\int_0^\infty p_s(x,t)ds\right]+T^s\int_0^{\infty}p_s(x,t)ds\\
=&-\frac{\partial}{\partial x}\left[F(x,t)\Phi_t w(x,t)\right]+\frac{\partial^2}{\partial x^2}\left[\frac{1}{2}\sigma^2(x,t)\Phi_t q(x,t)\right]+T^s\Phi_t q(x,t)\\
=&\left[-\frac{\partial}{\partial x} F(x,t)+\frac{1}{2}\frac{\partial^2}{\partial x^2}\sigma^2(x,t) +T^s\right]\Phi_t q(x,t)
\end{aligned}
\end{equation}

\end{proof}

\begin{rk}
\eqref{stable} is a special case of \eqref{symmetricFFPE} when L\'evy measure in Theorem \ref{thmsymmetric} is defined as in Definition \ref{stable-sde}.
\end{rk}

\subsection{Fractional Fokker-Planck equation with a general L\'evy generator}

Now we solve fractional Fokker-Planck equation associated with a general L\'evy process.
\begin{defin} \label{definitionoflevy}
Let $L(t)$ be a $ L\acute{e}vy \  process$ satisfying the following SDE
\begin{equation}
dL(t)=\int_B x\tilde{N} (dt,dx)+\int_{B^c} xN(dt,dx)
\end{equation} with
L\'evy symbol $\eta(u)=\int_{\mathbf{R}} [e^{iux}-1-iux \mathbbm{1}_B(x)]\nu(dx)$, where $\nu$ is a  $L\acute{e}vy $ measure.\\

\end{defin}

\begin{tm}\label{levythm}
Suppose that the standard Brownian motion $B(t)$,  the L\'evy process $L(t)$ as defined in definition \ref{definitionoflevy} and subordinator $T_\Psi(t)$ are independent, where $T_\Psi(t)$ has Laplace exponent $\Psi(u)$ and its inverse $S_\Psi(t)$. Let $Y(t)$ be the solution of the stochastic equation

\begin{equation}
\begin{aligned}
dY(t)&=F(Y^-(t),T^-_\Psi(t))dt+\sigma(Y^-(t),T^-_\Psi(t))dB(t)+h(Y^-(t),T^-_\Psi(t))dL(t), t\geq 0,\\
Y(0)&=0,\ T_\Psi(0)=0
\end{aligned}
\end{equation}
where the function $F(x,t),\ \sigma(x,t)\in C^2(\mathbf{R}^2),\ h(x,t)\in C^\infty(\mathbf{R}^2)$ satisfy the Lipschitz condition.
Assume that the PDF of the process $(Y(t),T_\Psi(T)),\ p_t(y,z) $ exists. Furthermore, we assume that $\frac{\partial}{\partial t}p_t(y,z)$, $\frac{\partial^k}{\partial y^k}p_t(y,z),\ k\in \mathbb{N}^+$ exist,
\begin{equation}
\int_0^{t_0}\int_0^{\infty}\left|\frac{\partial}{\partial t}p_s(x,t)\right|dsdt<\infty
\end{equation}
for each $t_0>0$,
\begin{equation}
\int_{x_1}^{x_2}\int_0^{\infty}\left|\frac{\partial}{\partial x}p_s(x,t)\right|dsdx<\infty,
\int_{x_1}^{x_2}\int_0^{\infty}\left|\frac{\partial^2}{\partial x^2}p_s(x,t)\right|dsdx<\infty,\\
\end{equation}
for each $x_1, x_2\in R$ and
\begin{equation}
\int_0^{\infty}\int_{\mathbf{R}-\{0\}}\sum_{k=1}^{\infty}\left|\frac{(-r)^k}{k!}\frac{\partial^k}{\partial x^k} (p_s(x,t)h(x,t)^k)\right|\nu(dr)ds<\infty.
\end{equation}
for each $x\in \mathbf{R}$.

Then, the PDF of the process $X(T)=Y^-(S_\Psi(t))$ is the weak solution of fractional Fokker-Planck equation
\begin{equation}\label{FFPE}
\frac{\partial}{\partial t}q(x,t)=\left[-\frac{\partial}{\partial x}F(x,t)+\frac{\partial^2}{\partial x^2}\frac{1}{2}\sigma^2(x,t)+T^{l+}\right]\Phi_t q(x,t),\\
\end{equation}
with $q(x,0)=\delta(x)$, where
\begin{equation}
T^{l+} f(x,t)=\int_{\mathbf{R}-\{0\}}\left[ \sum_{k=1}^{\infty}\frac{(-r)^k}{k!}\frac{\partial^k}{\partial x^k}(h(x,t)^k f(x,t) )+\mathbbm{1}_B(r,t)r\frac{\partial}{\partial x}(h(x,t)f(x,t))\right]\nu(dr),
\end{equation}
for any $f(x,t)\in C_0^\infty(R^2)$.
\end{tm}

\begin{rk}
The operator $T^{l+} $ appears naturally in Sun and Duan \cite{sun}. Theorem 2-14 in \cite{bit} guarantees the existence of density $p_t(y,z)$ by requiring some regularity on coefficients as follows, $F(x,t)$, $\sigma(x,t)$, $h(x,t)\in C^3_b(\mathbf{R}^2)$ have bounded partial derivatives from order 0 to 3, $\sup\limits_x|D^n_x(xh(x,t))|\in L^p(\mu,\nu)$ for all $p\geq2$.
\end{rk}

\begin{proof}[Proof of Theorem \ref{levythm}]
We follow the same steps as in the proof of Theorem \ref{thmstable}, and modify the part related to $T_\alpha$.

Let  $L=(L(t),t\geq 0)$ be a L\'evy process as defined in definition \ref{definitionoflevy}, it has infinitesimal generator
\begin{equation}
T^l f(y,z)=\int_{\mathbf{R}-\{0\}}\left[f(y+xh(y,z),z)-f(y,z)+\mathbbm{1}_B(x,z)xh(y,z)\frac{\partial}{\partial y}f(y,z)\right]\nu(dx)
\end{equation}

By equation (32) in \cite{sun},
\begin{equation}
T^{l+}p_t(y,z)=\int_{\mathbf{R}-\{0\}}\left[ \sum_{k=1}^{\infty}\frac{(-x)^k}{k!}\frac{\partial^k}{\partial y^k}(h(y,z)^kp_t(y,z)) +\mathbbm{1}_B(x,z)x\frac{\partial}{\partial y}(h(y,z)p_t(y,z))\right]\nu(dx)
\end{equation}

Since $\Gamma^+=A^++T^{l+}$ and $\frac{\partial}{\partial t}p_t(y,z)=\Gamma^+p_t(y,z)$, we have
\begin{equation}
\begin{aligned}
\frac{\partial}{\partial t}p_t(y,z)=&-\frac{\partial}{\partial y}(F(y,z)p_t(y,z))+\frac{\partial^2}{\partial y^2}\left[\frac{1}{2}\sigma^2(y,z)p_t(y,z)\right]-\frac{\partial}{\partial z}\Theta_z p_t(y,z)\\
&+\int_{\mathbf{R}-\{0\}}\left[ \sum_{k=1}^{\infty}\frac{(-x)^k}{k!}\frac{\partial^k}{\partial y^k}(h(y,z)^kp_t(y,z)) +\mathbbm{1}_B(x,z)x\frac{\partial}{\partial y}(h(y,z)p_t(y,z))\right]\nu(dx)
\end{aligned}
\end{equation}

Similar as proof of Theorem \ref{thmstable}, we can derive \eqref{equ1} and \eqref{equ2}, combined with $\lim\limits_{s\rightarrow \infty} p_s(x,t)=0$ and $p_0(x,t)=\mathbbm{1}_{(0,0)}(x,t)$,

\begin{equation}
\begin{aligned}
\frac{\partial}{\partial t}w(x,t)=&\int_0^\infty\left[ (-\frac{\partial}{\partial x}(F(x,t)p_s(x,t))+\frac{\partial^2}{\partial x^2}\left(\frac{1}{2}\sigma^2(x,t)p_s(x,t)\right)-\frac{\partial}{\partial s}p_s(x,t)\right.\\
&\left.+\int_{\mathbf{R}-\{0\}}\left( \sum_{k=1}^{\infty}\frac{(-r)^k}{k!}\frac{\partial^k}{\partial x^k}(h(x,t)^kp_s(x,t)) +\mathbbm{1}_B(r,t)r\frac{\partial}{\partial x}(h(x,t)p_s(x,t))\right)\nu(dr)\right]ds\\
=&-\frac{\partial}{\partial x}\left[F(x,t)\int_0^{\infty} p_s(x,t)ds\right]+\frac{\partial^2}{\partial x^2}\left[\frac{1}{2}\sigma^2(x,t)\int_0^\infty p_s(x,t)ds\right]\\
&+\int_{\mathbf{R}-\{0\}}\left[ \sum_{k=1}^{\infty}\frac{(-r)^k}{k!}\frac{\partial^k}{\partial x^k}\left(h(x,t)^k \int_0^{\infty}p_s(x,t)ds\right) +\mathbbm{1}_B(r,t)r\frac{\partial}{\partial x}\left(h(x,t)\int_0^{\infty}p_s(x,t)ds\right)\right]\nu(dr)\\
=&\left[-\frac{\partial}{\partial x}F(x,t)+\frac{\partial^2}{\partial x^2}\frac{1}{2}\sigma^2(x,t)\right]\Phi_t q(x,t)\\
&+\int_{\mathbf{R}-\{0\}}\left[ \sum_{k=1}^{\infty}\frac{(-r)^k}{k!}\frac{\partial^k}{\partial x^k}(h(x,t)^k \Phi_t q(x,t) )+\mathbbm{1}_B(r,t)r\frac{\partial}{\partial x}(h(x,t)\Phi_t q(x,t))\right]\nu(dr)\\
=&\left[-\frac{\partial}{\partial x}(F(x,t)+\frac{\partial^2}{\partial x^2}\frac{1}{2}\sigma^2(x,t)+T^l\right]\Phi_t w(x,t)
\end{aligned}
\end{equation}
\end{proof}

\begin{rk}\label{reduced}
In Theorem \ref{levythm}, when the space-time-dependent coefficient $h(x,t)$ of pure jump L\'evy process only depends on time $t$, say $h(t)$, then
\begin{equation}
\begin{aligned}
T^{l+}\Phi_t q(x,t)&=\int_{\mathbf{R}-\{0\}}\left[ \sum_{k=1}^{\infty}\frac{(-r)^k}{k!}\frac{\partial^k}{\partial x^k}(h(t)^k \Phi_t q(x,t ))+\mathbbm{1}_B(r,t)r\frac{\partial}{\partial x}(h(t)\Phi_t q(x,t))\right]\nu(dr)\\
&=\int_{\mathbf{R}-\{0\}}\left[ \sum_{k=1}^{\infty}\frac{(-rh(t))^k}{k!}\frac{\partial^k}{\partial x^k}(\Phi_t q(x,t ))+\mathbbm{1}_B(r,t)rh(t)\frac{\partial}{\partial x}\Phi_t q(x,t)\right]\nu(dr)\\
&=\int_{\mathbf{R}-\{0\}}\left[ \Phi_t q(x-rh(t),t)-\Phi_t q(x,t)+\mathbbm{1}_B(r,t)rh(t)\frac{\partial}{\partial x}\Phi_t q(x,t)\right]\nu(dr),
\end{aligned}
\end{equation}
thus, \eqref{FFPE} becomes
\begin{equation}
\begin{aligned}
\frac{\partial}{\partial t}q(x,t)=&\left[-\frac{\partial}{\partial x}(F(x,t)+\frac{\partial^2}{\partial x^2}\frac{1}{2}\sigma^2(x,t)\right]\Phi_t q(x,t),\\
&+\int_{\mathbf{R}-\{0\}}\left[ \Phi_t q(x-rh(t),t)-\Phi_t q(x,t)+\mathbbm{1}_B(r,t)rh(t)\frac{\partial}{\partial x}\Phi_t q(x,t)\right]\nu(dr),
\end{aligned}
\end{equation}
which corresponds \eqref{magequ}.
\end{rk}

\begin{rk}
The methods that Magdziarz and Zorawik used to calculate the adjoint of infinitesimal generator of the L\'evy process when the coefficient of the L\'evy noise depends only on time is substitution and integration by parts, (2.11) in \cite{mazo}. Such a method does not work when coefficient of L\'evy noise depends on both time and space; however, using the self-adjointness of the infinitesimall generator of symmetric L\'evy process and the method in \cite{sun} for general L\'evy process, we can figure out the adjoint operator for L\'evy noise with space-time-dependent coefficients. Thus, Theorem \ref{levythm} extends \cite{mazo}  and provides stochastic solution of fractional Fokker-Plank equation \eqref{FFPE} describing subdiffusion in full generality.
\end{rk}

\begin{rk}
Simulations of paths of stochastic processes play an important role in applications. Results in this paper provide a useful way for obtaining approximate solutions of fractional Fokker-Planck equations mentioned above.
Using Monte Carlo methods based on realization of $X(t)$, our results can be used to approximate solutions of fractional Fokker-Planck equations \eqref{stable}, \eqref{symmetricFFPE}, and \eqref{FFPE}, see \cite{jawe},  \cite{magd}, \cite{magdl}, \cite{pot}. Also our results can be used to obtain solution of equations with particle tracking methods, see \cite{beme}, \cite{chme}, \cite{mezh}.
\end{rk}

\end{document}